\documentclass[preprint,12pt]{elsarticle}
\usepackage{epsf}
\usepackage{amssymb,latexsym,amsmath,amscd,array,graphicx,amsthm}
\usepackage{mathtools}
\usepackage{hyperref}

\swapnumbers
\numberwithin{equation}{section}

\makeatletter
\def\ps@pprintTitle{%
   \let\@oddhead\@empty
   \let\@evenhead\@empty
   \def\@oddfoot{\reset@font\hfil}%
   \def\@evenfoot{\reset@font\hfil}%
}
\makeatother

\newcommand{\R}{\mathbb{R}}
\newcommand{\Q}{\mathbb{Q}}

\newcommand{\N}{\mathbb{N}}
\newcommand{\id}{\mathrm{id}}
\newcommand{\im}{\mathrm{im}}

\newcommand{\interior}{\mathrm{int}}

\newcommand{\sky}{\mathrm{sky}}
\newcommand{\grad}{\mathrm{grad}}

\theoremstyle{plain}
\newtheorem{thm}{Theorem}[section]
\newtheorem{cor}[thm]{Corollary}
\newtheorem{lem}[thm]{Lemma}
\newtheorem{prop}[thm]{Proposition}
\newtheorem{conjec}[thm]{Conjecture}

\newtheorem{result}[thm]{Result}

\theoremstyle{definition}
\newtheorem{defin}[thm]{Definition}

\newtheorem{rem}[thm]{Remark}
\newtheorem{example}[thm]{Example}

\newcommand\theoref{Theorem~\ref}

\newcommand\lemref{Lemma~\ref}
\newcommand\remref{Remark~\ref}
\newcommand\propref{Proposition~\ref}
\newcommand\corolref{Corollary~\ref}
\newcommand\defref{Definition~\ref}
\newcommand\conjref{Conjecture~\ref}

\journal{Journal of Geometry and Physics}

\begin{document}

\begin{frontmatter}

\title{Topological consequences of null-geodesic refocusing and applications to $Z^x$ manifolds}
\author[inst1]{Friedrich Bauermeister}
\ead{friedrich.bauermeister.gr@dartmouth.edu}
\affiliation[inst1]{organization={Department of Mathematics}, institution={Dartmouth College}}

\begin{abstract}
Let $(M,h)$ be a connected, complete Riemannian manifold, $x\in M$, and $l>0$. Then $M$ is called a $Z^x$ manifold if all geodesics starting at $x$ return to $x$, and it is called a $Y^x_l$ manifold if every unit-speed geodesic starting at $x$ returns to $x$ at time $l$. It is unknown whether there are $Z^x$ manifolds that are not $Y^x_l$ manifolds for any $l>0$. By the Bérard-Bergery theorem, any $Y^x_l$ manifold of dimension at least $2$ is compact with finite fundamental group. We prove the same result for $Z^x$ manifolds $M$ for which all unit-speed geodesics starting at $x$ return to $x$ in uniformly bounded time. We also prove that any $Z^x$ manifold $(M,h)$ with $h$ analytic is a $Y^x_l$ manifold for some $l>0$. We start by defining a class of globally hyperbolic spacetimes (called observer-refocusing) such that any $Z^x$ manifold is the Cauchy surface of some observer-refocusing spacetime. We then prove that under suitable conditions the Cauchy surfaces of observer-refocusing spacetimes are compact with finite fundamental group, and we show that analytic observer-refocusing spacetimes of dimension at least $3$ are strongly refocusing. We end by stating a contact-theoretic conjecture analogous to our results in Riemannian and Lorentzian geometry.
\end{abstract}

\end{frontmatter}

\section{Introduction}
Throughout the paper, all Riemannian manifolds and Lorentzian manifolds will be assumed to be smooth, connected, and without boundary. Further, all geodesics will be assumed to be maximal unless otherwise stated. If $(X,g)$ is a Lorentzian manifold, $p \in X$ and $v \in T_p X$, then $\alpha_v$ will always denote the geodesic defined through $\alpha_v(0)=p$ and $\alpha_v'(0)=v$.

\begin{defin}[$Y^{(x,y)}_l$ manifold]
Let $(M,h)$ be a complete Riemannian manifold, $x,y \in M$, and $l>0$. We say that $(M,h)$ is a \textit{$Y^{(x,y)}_l$ manifold} if for all geodesics $\alpha$ in $(M,h)$ with $\alpha(0)=x$ and $h(\alpha'(0),\alpha'(0))=1$ we have $\alpha(l)=y$. If $x=y$, we call $(M,h)$ a \textit{$Y^x_l$ manifold}. Note also that every $Y^{(x,y)}_l$ manifold is a $Y^x_{2l}$ manifold (see \cite{ChernovRudyak2008}).
\end{defin}

\begin{defin}[$Z^{(x,y)}$ manifold]
Let $(M,h)$ be a complete Riemannian manifold and $x,y \in M$. We say that $(M,h)$ is a \textit{$Z^{(x,y)}$ manifold} if for all geodesics $\alpha$ in $(M,h)$ with $\alpha(0)=x$ there is some $t>0$ with $\alpha(t)=y$. If $x=y$ we call $(M,h)$ a \textit{$Z^x$ manifold}. In contrast to $Y^{(x,y)}_l$ manifolds, it is not known whether every $Z^{(x,y)}$ manifold is a $Z^x$ manifold.
\end{defin}

\begin{defin}[strongly refocusing spacetime]\label{def:StronglyRefocusingSpacetime}
Let $(X,g)$ be a spacetime and let $p,q \in X$ with $p\neq q$. We say that $(X,g)$ is \textit{strongly refocusing with respect to $p$ and $q$} if all null-geodesics $\alpha$ in $(X,g)$ going through $p$ also go through $q$. We say that a spacetime $(X,g)$ is \textit{strongly refocusing} if there are $p\neq q$ such that $(X,g)$ is strongly refocusing with respect to $p$ and $q$.
\end{defin}

\begin{defin}[observer-refocusing spacetime]
Let $(X,g)$ be a spacetime, $p\in X$ and $\gamma$ some timelike curve in $X$. We say that $(X,g)$ is \textit{future (resp. past) observer-refocusing with respect to $p$ and $\gamma$} if for all future-oriented (resp. past-oriented) null-geodesics $\alpha$ in $(X,g)$ with $\alpha(0)=p$ there is some $t>0$ with $\alpha(t) \in \im(\gamma)$. We say a spacetime is \textit{future (resp. past) observer-refocusing} if it is future (resp. past) observer-refocusing with respect to some point $p\in X$ and some timelike curve $\gamma$ in $X$. 
\end{defin}

Future observer-refocusing spacetimes have the following physical interpretation. The worldline of physical observers through spacetime is a future-directed timelike curve. A spacetime is observer-refocusing if there is some point in spacetime and some observer, such that if light is emitted from the point in spacetime in every direction, the observer will, over the course of their life, be able to see all that light. Every $Y^{(x,y)}_l$ manifold is a $Z^{(x,y)}$ manifold and every strongly refocusing spacetime is a future observer-refocusing spacetime and a past observer-refocusing spacetime. It turns out that the roles of $p$ and $q$ in \defref{def:StronglyRefocusingSpacetime} are symmetric, see \cite{Kinlaw2011}. Without loss of generality, we may assume that $q \in J^+(p)$. By picking $\gamma$ to be any timelike curve through $q$, the spacetime will be future observer-refocusing with respect to $p$ and $\gamma$. Conversely, if we pick $\gamma$ to be any timelike curve through $p$, then the spacetime will be past observer-refocusing with respect to $q$ and $\gamma$. Throughout the rest of the paper we will use "observer-refocusing" to mean future observer-refocusing. Analogous results for past observer-refocusing spacetimes can be obtained by reversing the time-orientation of the spacetime.

\begin{defin}[spacetime associated to a Riemannian manifold]
Let $(M,h)$ be a Riemannian manifold. Let $X=M\times \R$ and let $g = h \oplus -dt^2$. Then $(X,g)$ is a Lorentzian manifold. We equip $X$ with the time-orientation given by the vectorfield $\partial_t$ to make it a spacetime. We call $(X,g)$ the \textit{spacetime associated to the Riemannian manifold $(M,h)$.} 
\end{defin}

\begin{defin}[globally hyperbolic spacetime, Cauchy surface]\label{def:GloballyHyperbolicSpacetime}
A spacetime $(X,g)$ is called \textit{causal} if it admits no closed timelike curves. A spacetime is called \textit{globally hyperbolic} if it is causal and if for all $p,q\in X$ the set $J^+(p) \cap J^-(q)$ is compact. A topological submanifold $S\subset X$ is called a \textit{topological Cauchy surface} of $(X,g)$ if every inextendible timelike curve in $X$ intersects $S$ exactly once. It is called an \textit{acausal} Cauchy surface if every inextendible causal curve in $X$ intersects $S$ exactly once. Smooth Cauchy surfaces which are spacelike are always acausal.
\end{defin}

A spacetime $(X,g)$ being globally hyperbolic is equivalent to the existence of a topological acausal Cauchy surface $S$ of $(X,g)$. Further, any two topological acausal Cauchy surfaces of a given spacetime $(X,g)$ are homeomorphic and a globally hyperbolic spacetime is homeomorphic to $S\times \R$ for any of its topological acausal Cauchy surfaces $S$; see \cite{Geroch1970}, \cite{HawkingEllis}. These results were later proven to also hold in the smooth category: A spacetime $(X,g)$ being globally hyperbolic is equivalent to the existence of a smooth, spacelike Cauchy surface $S$ of $(X,g)$. Further, any two smooth Cauchy surfaces of a given spacetime $(X,g)$ are diffeomorphic and a globally hyperbolic spacetime is diffeomorphic to $S\times \R$ for any of its smooth acausal Cauchy surfaces $S$; see \cite{BernalSánchez2003}. In the rest of this paper, when we talk about Cauchy surfaces, we will always mean smooth, spacelike Cauchy surfaces.

\begin{rem}
When the splitting theorems mentioned above were proven, the definition of a globally hyperbolic spacetime included the requirement that $(X,g)$ be \textit{strongly causal} instead of just causal. It was shown in \cite{BernalSánchez2007} that requiring causality suffices. 
\end{rem}

The existence of Cauchy surfaces ensures that the spacetime is well-behaved and that there are no naked singularities. It turns out that globally hyperbolic spacetimes are exactly the spacetimes obeying the cosmic censorship hypothesis as conjectured by Penrose in \cite{Penrose1998}, making them physically reasonable. 

\begin{defin}[temporal functions]
For a spacetime $(X,g)$, a \textit{temporal function} $\mathcal{T}:X\to \R$ is a smooth function with everywhere past-pointing timelike gradient. A surjective temporal function $\mathcal{T}:X\to \R$ is called \textit{Cauchy} if $\mathcal{T}^{-1}(t)\subset X$ is a Cauchy surface for all $t\in \R$. In \cite{BernalSánchez2003, BernalSánchez2004,BernalSánchez2005} Bernal and Sánchez prove that for every globally hyperbolic spacetime $(X,g)$ there are surjective Cauchy temporal functions $\mathcal{T}: X \to \R$. These results are strengthened and extended further in \cite{BernalSánchez2007}.
\end{defin}

Smooth temporal functions are strictly increasing along every future-directed causal curve. For globally hyperbolic spacetimes $(X,g)$, we can use a surjective Cauchy temporal function to express our metric $g$ in the following useful way.

\begin{rem}\label{rem:SplittingWithNoCrossTerms}
If $(X,g)$ is globally hyperbolic, we can take a surjective Cauchy temporal function $\mathcal{T}:X\to \R$. Let $M=\mathcal{T}^{-1}(0)$. For $p\in X$, let $\gamma_p$ be the maximal integral curve of $\grad(\mathcal{T})$ through $p$. Since $\gamma_p$ is an inextendible causal curve and since $M$ is a Cauchy surface, they intersect in a unique point $x_p\in M$. We define a diffeomorphism $\varphi: X\to M \times \R$ through $p\mapsto (x_p,\mathcal{T}(p))$. We can now identify $X = M\times \R$ through $\varphi$. Under this identification $\mathcal{T}^{-1}(t) = M \times \{t\}$ is a spacelike Cauchy surface for each $t\in \R$. Note that for $v\in T_{(x,t)}(M\times \{t\})$ we have $g(v,\grad(\mathcal{T})) = d\mathcal{T}(v)=0$. This shows that on the splitting $X=M\times \R$ the metric $g$ decomposes as $g=h-A dt^2$, where $A$ is some smooth, positive function, where $h$ is a symmetric $2$-tensor with $h(\partial_t,\cdot)=0$ and so that $h$ restricts to a Riemannian metric on $M\times \{t\}$ for each $t\in \R $.
\end{rem}

\begin{rem}
If $(M,h)$ is a complete Riemannian manifold, then the associated spacetime $(X,g)$ is globally hyperbolic, geodesically complete, and for each $t\in \R$ the submanifold $M \times \{t\} \subset X$ is a smooth spacelike Cauchy surface of $X$.
\end{rem}

\begin{prop}
If $(M,h)$ is a $Y^{(x,y)}_l$ manifold, then the associated spacetime $(X,g)$ is a strongly refocusing spacetime with $p=(x,0)$ and $q=(y,l)$. If $(M,h)$ is a $Z^{(x,y)}$ manifold, then the associated spacetime $(X,g)$ is an observer-refocusing spacetime with $p=(x,0)$ and observer $\gamma:\R \to X,$ $\gamma(t)=(y,t)$. 
\end{prop}
\begin{proof}
The results follow immediately after realizing that, up to orientation-preserving, affine reparametrization, any future-oriented null-geodesic $\alpha$ in $(X,g)$ through $(x,0)$ is given by $\alpha(t)=(\beta(t),t)$ where $\beta$ is a unit-speed geodesic through $x$ in $(M,h)$. 
\end{proof}

By proving theorems about the topology of Cauchy surfaces of globally hyperbolic observer-refocusing (resp. strongly refocusing) spacetimes, we will recover theorems about the topology of $Z^{(x,y)}$ (resp. $Y^{(x,y)}_l$) manifolds as a special case. To pass from statements about observer-refocusing spacetimes to statements about $Z^{(x,y)}$ manifolds, we will consistently consider the spacetime $(X,g)$ associated to a Riemannian $Z^{(x,y)}$ manifold $(M,h)$ and use observers of type $\gamma:(a,b)\to X=M\times \R$, $\gamma(t)=(y,t)$, so-called `standard observers' in the spacetime associated to the Riemannian manifold.

\section{Main Results}
We briefly state the main results of this paper.

\begin{result}[\theoref{thm:MainResult}]
Let $(X,g)$ be a globally hyperbolic spacetime of dimension at least $3$, let $p\in X$, let $L$ be some lightsphere at $p$ (see \defref{def:lightsphere}), let $\{\gamma_i\}_{i\in I}$ be some countable family of timelike curves and let $T>0$. If for each future-oriented null-geodesic $\alpha$ with $\alpha(0)=p$ and $\alpha'(0) \in L$ there is some $t \in (0,T]$ and $i \in I$ such that $\alpha(t) \in \im(\gamma_i)$, then $(X,g)$ has a compact Cauchy surface and finite fundamental group.
\end{result}

\begin{result}[\corolref{cor:SomeObserverRefocusingSpacetimesHaveCompactCauchySurfaceWithFiniteFundamentalGroup}]
Let $(X,g)$ be a globally hyperbolic spacetime of dimension at least $3$ which is observer-refocusing with respect to a point $p\in X$ and a timelike curve $\gamma$ such that the domain of $\gamma$ is compact. Then $(X,g)$ has a compact Cauchy surface and finite fundamental group.
\end{result}

\begin{result}[\corolref{cor:SomeZxManifoldsAreCompactWithFiniteFundamentalGroup}]
Let $(M,h)$ be a $Z^{(x,y)}$ manifold of dimension at least $2$ such that all unit-speed geodesics starting at $x$ reach $y$ in uniformly bounded time. Then $M$ is compact with finite fundamental group. 
\end{result}

\begin{result}[\theoref{thm:AnalyticMetricImpliesObserverRefocusingIsStronglyRefocusing}]
Let $(X,g)$ be a globally hyperbolic, observer-refocusing spacetime of dimension at least $3$ with an analytic metric $g$. Then $(X,g)$ is strongly refocusing.
\end{result}

\begin{result}[\corolref{cor:AnalyticMetricImpliesZxManifoldsAreYxl}, \theoref{thm:AppendixVersionAnalyticMetricImpliesZxManifoldsAreYxl}]\label{result:AnalyticMetricImpliesZxManifoldsAreYxl}
Let $(M,h)$ be a $Z^{(x,y)}$ manifold of dimension at least $2$ with an analytic metric $h$. Then $(M,h)$ is a $Y^{(x,y)}_l$ manifold for some $l>0$. 
\end{result}

Result \ref{result:AnalyticMetricImpliesZxManifoldsAreYxl} is proven once as a corollary of the Lorentzian result \theoref{thm:AnalyticMetricImpliesObserverRefocusingIsStronglyRefocusing}, and, because it might be of independent interest to Riemannian geometers less familiar with Lorentzian geometry, proven once more without using Lorentzian techniques in the appendix.

\section{Related results about $Y^x_l$ manifolds and refocusing spacetimes}

We first review results about the topology of $Y^x_l$ manifolds and refocusing spacetimes. $Y^x_l$ manifolds, as well as $Z^x$ manifolds and a whole family of related notions, are discussed in \cite[Definitions 7.7]{Besse}. The following result is due to Bott \cite{Bott1954} and Samelson \cite{Samelson1963}.

\begin{thm}[Bott-Samelson]
Let $(M,h)$ be a complete Riemannian manifold of dimension at least $2$, $x \in M$ and $l>0$ such that all unit-speed geodesics $\gamma$ with $\gamma(0)=x$ are periodic with period $l$. Then $\pi_1(M)$ is finite and the integral cohomology ring of $M$ is generated by one element. 
\end{thm}

In \cite{BerardBergery1977} Bérard-Bergery proves a similar statement only requiring the geodesics to revisit $x$ after a common time $l$, removing the requirement that the geodesics return with the same velocity they started with.

\begin{thm}[Bérard-Bergery]
Let $(M,h)$ be a $Y^x_l$ manifold of dimension at least $2$. Then $M$ is compact, has finite fundamental group and the rational cohomology ring of $M$ is generated by one element. 
\end{thm}

By recent contact Bott-Samelson Theorem results of Frauenfelder, Labrousse and Schlenk \cite{FrauenfelderLabrousseSchlenk2015}, the conclusion of the Bérard-Bergery theorem can be strengthened to be exactly the same as the one in the Bott-Samelson Theorem.

On the Lorentzian side, in \cite{Low2006}, Low defines a weaker version of a strongly refocusing spacetime.

\begin{defin}[refocusing spacetime]
Let $(X,g)$ be a Lorentzian manifold. We say that $(X,g)$ is \textit{refocusing} if there exists some $p \in X$ and some open set $U \ni p$ such that for all $V$ with $p\in V \subset U$ there is some $q \in X\setminus V$ such that all null-geodesics through $q$ enter $V$. 
\end{defin}

Low proved that refocusing globally hyperbolic spacetimes must have a compact Cauchy surface. In \cite[Theorem 11.5]{ChernovRudyak2008} Chernov and Rudyak prove that globally hyperbolic refocusing spacetimes of dimension at least $3$ have a globally hyperbolic semi-Riemannian universal cover which is also refocusing. They use this to prove the following result.

\begin{thm}
Every refocusing, globally hyperbolic spacetime of dimension at least $3$ has a compact Cauchy surface and finite fundamental group.
\end{thm}

In \cite{ChernovKinlawSadykov2010} Chernov, Kinlaw and Sadykov define a Riemannian analogue of refocusing spacetimes. 

\begin{defin}[$\tilde{Y}^x$ manifolds]
Let $(M,h)$ be a complete Riemannian manifold and let $x\in M$. We say that $M$ is a $\tilde{Y}^x$ manifold if there exists some $\bar{\varepsilon}>0$ such that for all $\varepsilon$ with $0<\varepsilon<\bar{\varepsilon}$ there is some $l>\varepsilon$ and some $y$ such that for all unit-speed geodesics $\gamma$ starting at $y$ we have $d(x,\gamma(l)) < \varepsilon$.
\end{defin}

\begin{prop}
If $(M,h)$ is a $\tilde{Y}^x$ manifold, then the associated spacetime $(X,g)$ is refocusing.
\end{prop}

The authors combine these results to show that every $\tilde{Y}^x$ manifold of dimension at least $2$ is compact with finite fundamental group. 

\begin{rem}
Just as it is unknown whether there are $Z^x$ manifolds that are not $Y^x_l$ manifolds \cite[Question 7.70]{Besse}, it is unclear whether there are $\tilde{Y}^x$ manifolds that are not $Y^x_l$ manifolds, whether there are refocusing spacetimes that are not strongly refocusing, and whether there are observer-refocusing spacetimes that are not strongly refocusing.
\end{rem}

It is tempting to define some notion of $\tilde{Z}^x$ manifold and a corresponding notion for spacetimes.

\begin{defin}[$\tilde{Z}^x$ manifolds]\label{def:tildeZx}
Let $(M,h)$ be a complete Riemannian manifold and let $x\in M$. We say that $(M,h)$ is a \textit{$\tilde{Z}^x$ manifold} if there exists some $\overline{\varepsilon}>0$ such that for all $\varepsilon$ with $0<\varepsilon<\overline{\varepsilon}$ and all unit-speed geodesics $\gamma$ starting at $x$ there is some $t>\varepsilon$ with $d(x,\gamma(t))<\varepsilon$. 
\end{defin}

\begin{defin}[very weakly refocusing spacetime]
Let $(X,g)$ be a spacetime and $p\in X$. We say that $(X,g)$ is \textit{very weakly refocusing} if every future-pointing null-geodesic starting at $p$ eventually enters the timelike future $I^+(p)$.
\end{defin}

Every $Z^x$ manifold is a $\tilde{Z}^x$ manifold, but the converse is false. For example, any $n$-torus $T^n$ with standard metric will be a $\tilde{Z}^x$ manifold, but only $T^1$ is a $Z^x$ manifold. We will show in \propref{prop:ObserverRefocusingWithGammaThroughPHasCompactCauchySurface} that globally hyperbolic very weakly refocusing spacetimes have a compact Cauchy surface. In fact every proof of compactness of Cauchy surfaces in this paper will ultimately rely on this fact. 

\begin{prop}
Let $(M,h)$ be a $\tilde{Z}^x$ manifold. Then the associated spacetime $(X,g)$ is very weakly refocusing. 
\end{prop}

\begin{proof}
Let $\alpha$ be a future-pointing null-geodesic with $\alpha(0)=(x,0)\eqqcolon p$. We can parametrize $\alpha$ so that $\alpha(t)=(\beta(t),t)$ for some unit-speed geodesic $\beta$ in $M$ with $\beta(0)=x$. Since $(M,h)$ is a $\tilde{Z}^x$ manifold, there is some $\varepsilon>0$ and some $T>\varepsilon$ such that $d(x,\beta(T))<\varepsilon$. We can assume that $d(x,\beta(T))\neq 0$ else increase $T$ slightly. Let $\tilde{\beta}$ be a unit-speed geodesic with $\tilde{\beta}(0)=x$ and $\tilde{\beta}(d(x,\beta(T)))=\beta(T)$. Consider the curve $\gamma$ in $X$ defined via $\gamma(t) = (\tilde{\beta}(t), \frac{\varepsilon}{d(x,\beta(T))} t)$. It is easy to see that $\gamma$ is a timelike curve with $\gamma(0)=p$ and $\gamma(d(x,\beta(T)))=(\beta(T),\varepsilon)$. Since $T>\varepsilon$ we have that $\alpha(T)=(\beta(T),T)$ is in the timelike future of $(\beta(T),\varepsilon)$, which is in the timelike future of $p=\alpha(0)=\gamma(0)$. This shows that $\alpha$ eventually enters $I^+(p)$. Since this is true for all future-pointing null-geodesics $\alpha$ starting at $p$, we conclude that $(X,g)$ is very weakly refocusing.
\end{proof}

\begin{rem}
We will prove in \corolref{cor:EveryTildeZxManifoldIsCompact} that every $\tilde{Z}^x$ manifold is compact. Further, taking any compact Riemannian manifold $(M,h)$ and any $x\in M$, the Poincaré recurrence theorem guarantees that almost all unit-speed geodesics starting at $x$ will exhibit the recurring behavior described in \defref{def:tildeZx}. It seems likely that every compact manifold $M$ can be equipped with some Riemannian metric such that there is some $x\in M$ making it into a $\tilde{Z}^x$ manifold. Maybe this is even true for a generic Riemannian metric on $M$. 
\end{rem}

In \cite{Low2006}, Low showed that the space of future-directed, unparameterized, maximal null-geodesics $\mathcal{N}$ of a strongly causal spacetime $(X,g)$ is a smooth manifold with a natural contact structure. For globally hyperbolic manifolds, there is a natural contactomorphism from $\mathcal{N}$ to $S^{\ast}M$ for any spacelike Cauchy surface $M\subset X$ given by the following map: To each future-directed null-geodesic $\alpha$ in $X$ there is a unique point $p\in M$ and $t\in \R$ with $\alpha(t)=p$. Let $\pi:X=M\times \R \to M$ be the natural projection. Consider the vector $v\in T_pM$ defined through $v=d\pi(\alpha'(t))$. Orientation-preserving affine reparametrization will change $v$ by multiplication with a positive scalar. Let $[\alpha]$ denote the equivalence class of null-geodesics up to reparametrization and let $[v]$ denote the equivalence class of vectors in $TM$ up to multiplication by a positive scalar. Then the map $\mathcal{N}\to SM$, $[\alpha]\mapsto [v]$ is well-defined. It is easy to check that this map is a diffeomorphism. Since $g$ restricts to a Riemannian metric on $M$, we can identify $SM$ with $S^{\ast}M$ through $g|_M$. This gives us the contactomorphism $\mathcal{N}\to S^{\ast}M$. For each $p\in X$, the set $\sky_p \coloneqq \{[\alpha] \in \mathcal{N} \mid \alpha \text{ is a null-geodesic through } p\}$ is called the \textit{sky} of $p$. Each sky is a Legendrian submanifold of $\mathcal{N}$, and if $M$ is any spacelike Cauchy surface through $p$ then $\sky_p$ corresponds to a fiber of $S^{\ast}M$ when identifying $\mathcal{N}\simeq S^{\ast}M$.

\begin{prop}\label{prop:NullGeodesicFlowInducesPositiveLegendrianIsotopy}
Let $(X,g)$ be a globally hyperbolic spacetime. Then $(X,g)$ splits as $X=M\times \R$ with metric $g=h - A dt^2$, as in \remref{rem:SplittingWithNoCrossTerms}. For $t\in \R$, let $h_t$ denote the Riemannian metric on $M\simeq M \times \{t\}$ given by $h_t = h|_{M\times \{t\}}$ and let $\alpha_t$ be the contact form on $(S^{\ast}M,\xi)$ induced by $h_t$. The co-orientations induced by each $\alpha_t$ agree. For each $t\in \R$, let $\phi_t:\mathcal{N}\to S^{\ast}(M\times \{t\})\simeq S^{\ast}M $ be the natural contactomorphism. Let $p\in X$. Then $(\phi_t(\sky_p))_{t\in \R}$ is a positive Legendrian isotopy in $S^{\ast}M$ (with respect to the co-orientation of any of the contact forms $\alpha_t$). 
\end{prop}
\begin{proof}
That $(\phi_t(\sky_p))_{t\in \R}$ is a Legendrian isotopy follows from the fact that each $\phi_t$ is a contactomorphism. It is left to show that the isotopy is positive. Let $[\beta]\in \sky_p$. We can parametrize $\beta$ as a pre-geodesic so that $\beta(t)=(x(t),t)$. Note that $\pi(\phi_t([\beta]))=x(t)$. Let $\psi:U_{h_t}M \to S^{\ast}M$ be the natural map from the unit tangent bundle to the spherical cotangent bundle. By definition of $\phi_t$, there is some $\lambda>0$ with $\psi^{-1}(\phi_t([\beta])) = \lambda x'(t)$. Therefore
\begin{align*}
\alpha_t\left(\frac{d}{dt}\phi_t([\beta])\right) &= h_t\left(\psi^{-1}(\phi_t([\beta])),d\pi\left(\frac{d}{dt}\phi_t([\beta])\right)\right)\\
&= h_t\left(\lambda x'(t),x'(t)\right) >0.
\end{align*}
\end{proof}

We can now prove a Bott-Samelson type result for the Cauchy surfaces of globally hyperbolic strongly refocusing spacetimes. 

\begin{prop}\label{prop:BottSamelsonForStronglyRefocusingSpacetime}
Let $X$ be a globally hyperbolic, strongly refocusing spacetime of dimension at least $3$. Then the Cauchy surface $M$ of $X$ is compact with finite fundamental group, and its universal cover has the integral cohomology ring of a CROSS.
\end{prop}
\begin{proof}
We can choose to write $X=M\times \R$ with metric $g=h-Adt^2$, as in \remref{rem:SplittingWithNoCrossTerms}. Since $(X,g)$ is strongly refocusing, there are $p,q \in X$, $p\neq q$ with $\sky_p=\sky_q$. Without loss of generality assume $q\in J^+(p)$. Then $p=(x,t_1)$ and $q=(y,t_2)$ for some $x,y\in M$ and $t_1,t_2\in \R$ with $t_1 < t_2$. Let $\sky\coloneqq \sky_p=\sky_q$. Define $\phi_t$ as in \propref{prop:NullGeodesicFlowInducesPositiveLegendrianIsotopy}. Then $(\phi_t(\sky))_{t\in [t_1,t_2]}$ is a positive Legendrian isotopy from $\phi_{t_1}(\sky)=S^{\ast}_xM$ to $\phi_{t_2}(\sky)=S^{\ast}_yM$. Since $S^{\ast}M$ is symmetric with respect to choice of fibers, there is also a positive Legendrian isotopy from $S_y^{\ast}M$ to $S_x^{\ast}M$. Concatenating these isotopies yields a non-negative Legendrian isotopy from $S^{\ast}_xM$ to itself. By \cite[Remark 1.1]{Chernov2018} this proves the claim.
\end{proof}

\begin{rem}
There is a different way to construct a non-negative Legendrian isotopy from one fiber of $S^{\ast}M$ to itself in the case where $M$ is the Cauchy surface of a globally hyperbolic strongly refocusing spacetime $(X,g)$. Let $p,q\in X$ have the same sky. By \cite[Proposition 1.1]{ChernovNemirovski2016}, \cite[Proposition 5.8]{BautistaIbortLafuente2014} there is a non-negative Legendrian isotopy between skies of any causally related points. Since $p$ and $q$ are causally related with the same sky, there is a non-negative Legendrian isotopy from $\sky_p$ to itself. Fixing some Cauchy surface $M\subset X$ with $p\in M$ and identifying $\mathcal{N}$ with $S^{\ast}M$ we get a non-negative Legendrian isotopy from a fiber of $S^{\ast}M$ to itself.
\end{rem}

The question about the cohomology of the universal cover of the Cauchy surface was an open question in \cite{ChernovKinlawSadykov2010}. The affirmative answer to it was known to Chernov and Nemirovski but not yet published.

\section{Observer-Refocusing Spacetimes}

In a Riemannian manifold, normalizing non-zero tangent vectors gives a canonical representative for each spacelike direction at a point. There is no canonical choice of representative for future-pointing null-directions at a point in a Lorentzian manifold, but it will often be useful to make \textit{some} (smooth) choice of representative for each null-direction.

\begin{defin}[lightsphere]\label{def:lightsphere}
Let $X$ be a Lorentzian manifold and $p \in X$. We denote by $\mathcal{L}_p$ the set of future-pointing null-directions at $p$, i.e. we define $\mathcal{L}_p= \{[v]_{\sim} \mid v \text{ future-pointing and null} \}$ where $\sim$ is the equivalence relation on the set of null-vectors defined through $v_1 \sim v_2 \iff \exists \lambda>0: v_1 = \lambda v_2$. Let $\mathcal{N}_pX$ be the space of null-vectors in $T_pX$. We call a smooth submanifold $L\subset \mathcal{N}_p X$ a \textit{lightsphere} at $p$ if the natural map from $L$ to $\mathcal{L}_p$ is a homeomorphism. 
\end{defin}

Such lightspheres always exist, for example one can take any auxiliary Riemannian metric $h$ on a Lorentzian manifold $(X,g)$ and define $L = \{ v \in T_p X \mid v \text{ future-pointing and null and } h(v,v)=1 \}$. 

\begin{lem}\label{lem:LightraysInFutureImpliesCompactCauchySurface}
Let $(X,g)$ be a globally hyperbolic spacetime which is very weakly refocusing with respect to some $p\in X$. Then the Cauchy surface of $(X,g)$ is compact.
\end{lem}

\begin{proof}
Let $L$ be some lightsphere at $p$. For each $v\in L$ let $t_v$ be such that $\alpha_v(t_v) \in I^+(p)$. Take some open neighborhood $U_v$ around $v$ such that for all $w\in \overline{U_v}$ we have that $\alpha_w(t_v) \in I^+(p)$. Since $L$ is compact, we can cover it with a finite number of such neighborhoods $U_{v_1},\dots,U_{v_n}$. Consider the set $K=\bigcup_{i=1}^n \{\alpha_w(t_{v_i}) \mid w \in \overline{U_{v_i}}\}$, which is compact as a finite union of compact sets. Let $S \subset X$ be a Cauchy surface with $K \subset I^-(S)$ and consider the set $J^+(p) \cap S$. Obviously $J^+(p) \cap S$ is a non-empty, closed subset of $S$. Further, we get that
$$\partial_S \left( J^+(p) \cap S \right) \subseteq \partial_X (J^+(p)) \cap S = \left( J^+(p) \setminus I^+(p) \right) \cap S = \emptyset,$$
where the last equality holds for the following reason. Let $q \in J^+(p) \setminus I^+(p)$. Then $q$ must lie on a future-oriented null-geodesic through $p$, i.e. $q = \alpha_v(t)$ for some $v \in L$ and some $t >0$. Let $i\in \{1,\dots,n\}$ be such that $v\in U_{v_i}$. Since we chose $q \notin I^+(p)$, we must have $t\leq t_{v_i}$. But this means that $q \in I^-(S)$, hence $q \notin S$. This shows that $J^+(p) \cap S$ has empty boundary in $S$, which means $J^+(p) \cap S = S$. Since $J^+(p) \cap S$ is always compact, we are done.
\end{proof}

\begin{cor}\label{cor:EveryTildeZxManifoldIsCompact}
If $(M,h)$ is a $\tilde{Z}^x$ manifold, then $M$ is compact. In particular, every $Z^x$ manifold is compact.
\end{cor}
\begin{proof}
This follows from \lemref{lem:LightraysInFutureImpliesCompactCauchySurface} by the fact that the spacetime associated to $(M,h)$ is very weakly refocusing.
\end{proof}

\begin{prop}\label{prop:ObserverRefocusingWithGammaThroughPHasCompactCauchySurface}
Let $(X,g)$ be a globally hyperbolic spacetime, $p\in X$ and $\gamma$ a timelike curve such that $(X,g)$ is observer-refocusing with respect to $p$ and $\gamma$. If $p=\gamma(0)$, then $(X,g)$ is very weakly refocusing. In particular $(X,g)$ has a compact Cauchy surface.
\end{prop}
\begin{proof}
For every future-directed null vector $v\in T_pX$ there is some $q\in \im(\gamma)$ and $t>0$ such that $\alpha_v(t) \in \im(\gamma)$, i.e. $\alpha_v(t)=\gamma(s)$ for some $s\in \R$. We have $\alpha_v(t) \in J^+(p) \setminus \{p\}$. Since $\gamma$ is timelike we must have $s>0$ and hence $\alpha_v(t)=\gamma(s) \in I^+(p)$. We apply \lemref{lem:LightraysInFutureImpliesCompactCauchySurface} and are done.
\end{proof}

\begin{lem}\label{lem:DenseSubsetOfLsuffices}
Let $(X,g)$ be a globally hyperbolic spacetime, $p\in X$, let $L$ be a lightsphere at $p$, let $L'\subseteq L$ be a dense subset of $L$, and let $T>0$ so that $\alpha_v(T)$ is defined for all $v\in L$. If for every null-geodesic $\alpha$ with $\alpha(0)=p$ and $\alpha'(0) \in A$ we have that $\alpha(t) \in I^+(p)$ for all $t>T$, then $X$ has a compact Cauchy surface.
\end{lem}
\begin{proof}
Let $v\in L$ and let $v_n$ be a sequence in $L'$ converging towards $v$. Let $t_n$ be such that for every $n\in \N$ we have that $\alpha_{v_n}(t_n)$ is the null-cut point of $\alpha_{v_n}|_{[0,t_n]}$. By assumption, the sequence $(t_n)_{n \in \N}$ is bounded by $T$. Without loss of generality, assume that the sequence converges to some value $t$, else go to a subsequence. Since the null-cut locus $C_N^+(p)$ is closed (see \cite[Proposition 6.5]{BeemEhrlich1979}), we have $\alpha_v(t) \in C_N^+(p)$. Therefore $\alpha_v(t') \in I^+(p)$ for all $t'>t$. The result follows by applying \lemref{lem:LightraysInFutureImpliesCompactCauchySurface}.
\end{proof}

\begin{lem}\label{lem:Ifexp(t_nv_n)ConvergesThen(t_n)IsBounded}
Let $X$ be a globally hyperbolic spacetime, $p\in X$, and $L$ an embedded lightsphere at $p$. If $(v_n,t_n)_{n\in \N}$ is a sequence in $L\times \R$ such that $\alpha_{v_n}(t_n)$ is always defined and such that $(\alpha_{v_n}(t_n))_{n \in \N}$ converges to some $q\in X$, then $(t_n)_{n \in \N}$ is bounded. 
\end{lem}
\begin{proof}
Let $\mathcal{T}:X\to \R$ be a smooth, surjective temporal function such that every level set of $\mathcal{T}$ is a Cauchy surface. Take some $s_1,s_2 \in \R$ such that $s_1 < \mathcal{T}(q) <s_2$. Since $X$ is globally hyperbolic and since level sets of $\mathcal{T}$ are Cauchy surfaces, for every $v\in L$ and $s\in [s_1,s_2]$, the equation $\mathcal{T}(\exp(tv)) = s$ is satisfied by a unique $t\in \R$. Since $\mathcal{T}$ has timelike gradient, we have $\frac{d}{dt} \mathcal{T}(\exp tv) \neq 0$ for all $t$. By the implicit function theorem, the function $f:L\times [s_1,s_2]\to \R$ implicitly defined through $\mathcal{T}(\exp(f(v)v))=s$ is smooth. Since $L\times [s_1,s_2]$ is compact, so is its image under $f$. Let $N\in \N$ be such that for all $n\geq N$ we have $\mathcal{T}(\alpha_{v_n}(t_n)) \in (s_1,s_2)$. This means that for all $n\geq N$ we have $t_n \in f(L\times [s_1 ,s_2])$, which is compact hence bounded.
\end{proof}

\begin{lem}\label{lem:int(intersection)subsetRefocusing}
Let $(X,g)$ be a globally hyperbolic spacetime of dimension at least $3$, let $p\in X$ and let $\gamma$ be some timelike curve in $X$ with $\im(\gamma)$ closed. Let $L$ be a lightsphere at $p$ and let $T>0$ be such that $\alpha_v(T)$ is defined for all $v\in L$. Let $A = \{v \in L \mid \alpha_v|_{(0,T]} \text{ intersects } \gamma\}$ and $B = \{v \in L \mid p \text{ has a conjugate point } q \in \im(\gamma) \text{ along } \alpha_v|_{(0,T]}\}$. Then $\interior(A) \subseteq B$.
\end{lem}

\begin{proof}
Let $v\in \interior(A)$. Since $\dim(X)\geq 3$ implies $\dim(L)\geq 1$, there is some sequence $(v_n)_{n\in \N}$ in $A$ converging to $v$. For each $n\in \N$ let $t_n \in (0,T])$ and $q_n \in \im(\gamma)$ be such that $\alpha_{v_n}(t_n)=q_n$. Since all $q_n$ are contained in the compact set $\{\exp(tv)\ \mid v\in L, t\in [0,T]\}$, we can assume that $(t_n)_{n\in \N}$ and $(q_n)_{n\in \N}$ converge to some $t$ and $q$ respectively, else go to convergent subsequences. Toward a contradiction assume that $d\exp_{tv}$ is not singular. Then there exists an open neighborhood $U$ of $tv$ in $\mathcal{N}_p$ such that $\exp|_U$ is a diffeomorphism onto its image, in particular $N\coloneqq \exp(U)\subset X$ is an embedded null hypersurface, i.e. an embedded hypersurface whose normal vectorfield is null. This implies that $TN\subset TX$ contains no timelike vectors, hence all intersections of $\gamma$ and $\exp(U)$ must be transverse, and therefore isolated. This is a contradiction to the fact that any neighborhood of $tv$ contains infinitely many null vectors $t_nv_n$ with $\exp(t_nv_n) \in \im(\gamma)$. We conclude that $d\exp_{tv}$ is singular. This means that $q = \exp(tv)$ is conjugate to $p$ along $\alpha_v|_{(0,T]}$, and hence $v\in B$.
\end{proof}

\begin{rem}\label{rem:wlogOurSpacetimeIsNullgeodesicallyComplete}
In general, the spacetimes $(X,g)$ we consider will not be null-geodesically complete, so expressions of the type $\alpha_v(t)$ might become undefined for large $t$. However, every globally hyperbolic spacetime $(X,g)$ is conformally equivalent to a spacetime $(X',g')$ which \textit{is} null-geodesically complete (see \cite{Clarke1971}, where this is proven for all strongly causal spacetimes). Conformal isomorphisms preserve both the causal structure and topology of our spacetime. Therefore, when appropriate, we will assume \textit{without loss of generality} that our spacetime $(X,g)$ is null-geodesically complete.
\end{rem}

\begin{rem}
Another way to handle the inconvenience of expressions of the type $\alpha_v(t)$ becoming undefined for large $t$ would be to parametrize null-geodesics using a surjective Cauchy temporal function $\mathcal{T}$, i.e. picking for each null-geodesic $\alpha_v$ a parametrization $t\mapsto\rho_v(t)$ such that $\alpha_v(t)\coloneqq \exp(\rho_v(t)v)$ with $\mathcal{T}(\alpha_v(t))=t$. The role of the uniform bound $T$ in, for example, \lemref{lem:DenseSubsetOfLsuffices} is then simply that the refocusing is happening inside a Cauchy slab $\mathcal{T}^{-1}([\mathcal{T}(p),T])$. We elect instead to keep the natural parametrization $\alpha_v(t)=\exp(tv)$, because it makes some calculations nicer.
\end{rem}

\begin{thm}\label{thm:MainResult}
Let $(X,g)$ be a globally hyperbolic spacetime of dimension at least $3$, let $p\in X$, let $L$ be some lightsphere at $p$, let $\{\gamma_i\}_{i\in I}$ be some countable family of timelike curves and let $T>0$. If for each future-oriented null-geodesic $\alpha$ with $\alpha(0)=p$ and $\alpha'(0) \in L$ there is some $t \in (0,T]$ and $i \in I$ such that $\alpha(t) \in \im(\gamma_i)$, then $(X,g)$ has a compact Cauchy surface and finite fundamental group. 
\end{thm}

\begin{proof}
By \remref{rem:wlogOurSpacetimeIsNullgeodesicallyComplete}, we can assume that $(X,g)$ is null-geodesically complete. Further, we can assume without loss of generality that the image of each $\gamma_i$ is compact, else subdivide the non-compact $\gamma_i$ into countably many overlapping compact timelike curves. For each $i\in I$ consider the set $A_i = \{v \in L \mid \alpha_v|_{(0,T]} \text{ intersects } \gamma_i \}$. Since the image of $\gamma_i$ is compact, each set $A_i$ is closed. By assumption $L = \bigcup_{i \in I} A_i$. By the Baire category theorem, $L'\coloneqq \bigcup_{i\in I} \interior(A_i)$ is dense in $L$. For all $v\in L'$, there is some $i\in I$ such that $v\in \interior(A_i)$. By \lemref{lem:int(intersection)subsetRefocusing} $p$ then has a conjugate point $q$ along $\alpha_v|_{(0,T]}$ and therefore $\alpha_v(t) \in I^+(p)$ for all $t>T$. Application of \lemref{lem:DenseSubsetOfLsuffices} yields that $(X,g)$ has a compact Cauchy surface. 

To show that the Cauchy surface of $(X,g)$ has finite fundamental group, we proceed as follows. Let $\rho: (\tilde{X},\tilde{g})\to(X,g)$ be the semi-Riemannian universal cover of $X$. Let $\tilde{p}$ be any point in $\rho^{-1}(p)$ and let $\tilde{L}$ be the lightsphere at $p$ given by $\tilde{L} = (d\rho)^{-1}(L)$. Let $\tilde{v} \in \tilde{L}$ and $v=d\rho(\tilde{v}) \in L$. Then there is some $t\in (0,T]$ and $i \in I$ with $\alpha_v(t) \in \im(\gamma_i)$. Lifting $\alpha_v$ to $\tilde{X}$, we get that $\alpha_{\tilde{v}}(t) \in \im(\tilde{\gamma_i})$ for some lift $\tilde{\gamma}_i$ of $\gamma_i$. Since each curve $\gamma_i$ has only countably many lifts, we can apply the first part of this theorem to the family $\{\tilde{\gamma} \text{ timelike curve in }\tilde{X}\mid (\rho \circ \tilde{\gamma}) \in \{\gamma_i\}_{i\in I}\}$. By applying the first part of this theorem, we can conclude that $\tilde{X}$ has a compact Cauchy surface as well. Now pick a Cauchy surface $S\subset X$ and let $\tilde{S} = \rho^{-1}(S)$. Then $\tilde{S}$ is a Cauchy surface of $\tilde{X}$ (see the proof of \cite[Theorem 11.5]{ChernovRudyak2008}) and hence compact. Further, $\tilde{S}$ is simply connected (as $\tilde{X}$ is simply connected and diffeomorphic to $\tilde{S} \times \R$), and $\rho: \tilde{S} \to S$ is a covering map. Pick any $x\in \tilde{S}$, then
$$\#\pi_1(S,x) = \# \rho^{-1}(x),$$
which is finite since $\tilde{S}$ is compact.
\end{proof}

\begin{example}\label{ex:Minkowski}
Consider $n+1$ dimensional Minkowski spacetime $X = \R^n \times \R^1_1$ and take some countable set of observers $\{\gamma_i\}_{i\in I}$ and some $p\in X$. Then the conditions of \theoref{thm:MainResult} are not fulfilled. For a null-geodesic $\alpha$ with $\alpha(0)=p$ an observer $\gamma$ can intersect at most once with $\alpha$, and so countably many observers can (together) see only countably many light rays. More explicitly, if one takes the set of $\{\gamma_i\}_{i\in I}$ to be the set of standard observers $\{t\mapsto (q,t)\}_{q\in \Q^n}$ and takes $p=(0,0)$, then all light rays with irrational initial conditions will escape. \theoref{thm:MainResult} gives a \textit{topological} reason for why this has to happen.
\end{example}

The following example shows that the conditions of \theoref{thm:MainResult} are sufficient but not necessary for a spacetime to have a compact Cauchy surface with finite fundamental group.

\begin{example}\label{ex:deSitter}
Consider the de Sitter spacetime of dimension $n+1$ ($n\geq 2$), whose Cauchy surfaces are topologically spheres of dimension $n$, so in particular are compact with finite fundamental group. We note that de Sitter space does not meet the conditions of \theoref{thm:MainResult}. If it did, there would have to be null-geodesics entering into their own chronological future, but this does not happen in de Sitter spacetimes. This makes de Sitter spacetime an example of a globally hyperbolic spacetime with compact Cauchy surfaces with finite fundamental group which is not even very weakly refocusing (and a fortiori not observer-refocusing, strongly refocusing, or refocusing).
\end{example}

\begin{cor}\label{cor:SomeObserverRefocusingSpacetimesHaveCompactCauchySurfaceWithFiniteFundamentalGroup}
Let $(X,g)$ be a globally hyperbolic spacetime of dimension at least $3$ which is observer-refocusing with respect to a point $p\in X$ and a timelike curve $\gamma$ such that the domain of $\gamma$ is compact. Then $(X,g)$ has a compact Cauchy surface and finite fundamental group. 
\end{cor}
\begin{proof}
By \remref{rem:wlogOurSpacetimeIsNullgeodesicallyComplete}, we can assume that $(X,g)$ is null-geodesically complete. Using similar arguments as in \lemref{lem:Ifexp(t_nv_n)ConvergesThen(t_n)IsBounded} compactness of $\im(\gamma)$ ensures that there is some lightsphere $L$ at $p$ and some $T>0$ such that for all $v\in L$ there is some $t\in (0,T]$ with $\alpha_v(t)\in \im(\gamma)$. Now we can use \theoref{thm:MainResult}.
\end{proof}

\begin{cor}\label{cor:SomeZxManifoldsAreCompactWithFiniteFundamentalGroup}
Let $(M,h)$ be a $Z^{(x,y)}$ manifold of dimension at least $2$ such that all unit-speed geodesics starting at $x$ reach $y$ in uniformly bounded time. Then $M$ is compact with finite fundamental group. 
\end{cor}

\begin{proof}
Let $T>0$ be the uniform bound on the time it takes unit-speed geodesics starting at $x$ to reach $y$. Consider the associated spacetime $(X,g)$ to $(M,h)$. Defining $p=(x,0)$ and the standard observer $\gamma:[0,T]\to X$, $\gamma(t)=(y,t)$, the result then follows from \corolref{cor:SomeObserverRefocusingSpacetimesHaveCompactCauchySurfaceWithFiniteFundamentalGroup}.
\end{proof}

\begin{cor}
Let $(M,h)$ be a $Z^{(x,y)}$ manifold of dimension $2$ or $3$, let $x\in M$ and assume that all unit-speed geodesics starting at $x$ reach $y$ in uniformly bounded time. Then there exists a Riemannian metric $\tilde{h}$ on $M$ such that $(M,\tilde{h})$ is a $Y^x_l$ manifold.
\end{cor}
\begin{proof} In this case $M$ is compact and has finite fundamental group. If $\dim(M)=3$, then by the Thurston Elliptization conjecture \cite{Thurston1997} proved by Perelman \cite{Perelman2002, Perelman2003, Perelman2003-2}, $M$ is a quotient of the sphere $S^3$ (with standard metric) by a finite group of isometries. The induced metric on $M$ will make it into a $Y^x_{2\pi}$ manifold. Similar arguments work for $\dim(M)=2$. 
\end{proof}

\begin{rem}\label{rem:WeCanGetopenDenseSetWithoutUniformBound}
One can carefully check the proof of \lemref{lem:int(intersection)subsetRefocusing} and \theoref{thm:MainResult} to see that if we drop the requirement of a uniform bound $T>0$ from \theoref{thm:MainResult} there will still be an open dense subset $L'\subset L$ such that all null-geodesics $\alpha_v$ starting at $p$ with initial velocity $v\in L'$ have a conjugate point along one of the curves $\gamma_i$, hence eventually enter $I^+(p)$. In the proof of \theoref{thm:MainResult} one can define the countable family $(A_{(i,N)})_{i\in I, N\in \N}$ through $A_{(i,N)} = \{v \in L \mid \alpha_v|_{(0,N]} \text{ intersects } \gamma_i \}$. Then the union $L'$ of the interiors of all the $A_{(i,N)}$ is open and dense in $L$. But this is not enough to conclude that we have a compact Cauchy surface since we cannot apply \lemref{lem:DenseSubsetOfLsuffices} without a uniform bound. 
\end{rem}

\begin{rem}
It is unknown whether there are $Z^x$ manifolds where the time-of-return of unit-speed geodesics is not uniformly bounded \cite[Question 7.71]{Besse}. Similarly, it is unknown whether there are any globally hyperbolic observer-refocusing spacetimes where the timelike curve cannot be taken to be compact. 
\end{rem}

\section{The Analytic Case}

\begin{lem}\label{lem:DerivativeOfCurveOfLightlikeVectorsCannotBeTimelike}
Let $(X,g)$ be a Lorentzian manifold and let $h:[0,1] \to T_pX$ be a curve such that $h(t)$ is null for each $t\in [0,1]$. Then $d\exp_{h(0)}(h'(0)) = \frac{d}{dt}|_{t=0}\exp(h(t))$ is not timelike.
\end{lem}
\begin{proof}
We have that $g(h(t),h(t)) = 0$ for all $t \in [0,1]$. Therefore $2 g(h(0),h'(0)) =\frac{d}{dt}|_{t=0} g(h(t),h(t)) = 0$. This means that $h'(0)$ is perpendicular to the null-vector $h(0)$. Let us denote $v=h(0)$ and $w=h'(0)$. By the Gauss lemma, 
$$g(d\exp_v(v),d\exp_v(w))=g(v,w)=0.$$
Here $d\exp_v(v)$ is a null-vector since $v$ is one. Therefore $d\exp_v(w)$ is orthogonal to a null-vector and hence cannot be timelike.
\end{proof}

\begin{thm}\label{thm:NullConjugateLocusIsNice}
Let $(X,g)$ be a globally hyperbolic spacetime. Let $p\in X$ and let $L$ be a lightsphere at $p$. Let $v\in L$ be such that the geodesic $\alpha_v$ has at least $k$ points conjugate to $p$. Then there is some neighborhood $U\subseteq L$ of $v$ such that for each $w\in U$, the null-geodesic $\alpha_w$ has at least $k$ conjugate points (counted with multiplicity). Let $0<\lambda_1(w) \leq \lambda_2(w)\leq \dots \leq \lambda_k(w)$ be the parameter values at which $\alpha_w$ hits a conjugate point. Then the functions $\lambda_j$ are continuous on $U$ and they are smooth on an open, dense subset of $U$.
\end{thm}
\begin{proof}
The existence of the neighborhood $U$ and the continuity of the functions $\lambda_j$ are essentially the content of \cite[Proposition 4.2]{Rosquist1983}. The fact that each $\lambda_j$ is smooth on an open dense set of $U$ follows from \cite[Theorem 4.3]{Rosquist1983}: The set $\{\lambda_j(v)v \mid v\in U\}$ is a subset of the null $T$-conjugate locus. Let $C\subseteq \mathcal{N}_pX$ denote the null $T$-conjugate locus. The regular null $T$-conjugate locus as defined in \cite{Rosquist1983} is an open dense subset $C_R\subseteq C$ in the null $T$-conjugate locus and it is a smooth submanifold of $\mathcal{N}_pX$ transverse to rays in $\mathcal{N}_pX$. For every $j\in \{1,\dots, k\}$ the set $V_j\coloneqq \{\lambda_j(v)v \mid v\in U\}$ is an open subset of the $C$, hence $V_j\cap C_R$ is open and dense in $V_j$. Since $V_j\cap C_R$ is also an open subset of $C_R$, it is in particular a smooth submanifold of $\mathcal{N}_pX$ transverse to the rays in $\mathcal{N}_pX$. For each $j$ let $U_j\coloneqq \{\pi(w)\mid w\in V_j\cap C_R\}$ where $\pi:T_xM\setminus\{0\}\to U$ is the radial projection. Each $U_j$ is open and dense in $U$. By construction $\{\lambda_j(v)v \mid v\in U_j\}=V_j\cap C_R$ is a smooth submanifold of $\mathcal{N}_pX$ transverse to the rays in $\mathcal{N}_pX$. By the implicit function theorem the map $v\mapsto \lambda_j(v)$ is smooth on $U_j$. The intersection of all the sets $U_j$ gives one open, dense subset of $U$ on which all functions $\lambda_j$ are smooth.
\end{proof}

\begin{lem}\label{lem:MapSelectingFirstConjugatePointIsLocallyConstantSomewhere}
Let $(X,g)$ be a globally hyperbolic spacetime of dimension at least $3$. Let $\gamma$ be some timelike curve in $X$, let $p\in X$ and let $L$ be a lightsphere at $p$. Let $U\subseteq L$ be some non-empty open set. Let $f:U\to \R^{+}$ be some lower semi-continuous function such that for all $v\in U$ we have that $\exp(f(v)v) \in \im(\gamma)$ is a point conjugate to $p$ along the geodesic $\alpha_v$. Then there is some non-empty open set $V\subseteq U$ on which $\exp(f(v)v)$ is constant.
\end{lem}
\begin{proof}
Let the functions $\lambda_j$ be as in \theoref{thm:NullConjugateLocusIsNice}. For $v\in U$ we have that $\exp(f(v)v)$ is a conjugate point along $\alpha_v$. Let $k(v)$ denote the minimal $k$ with $\lambda_k(v)=f(v)$. We will show that $v\mapsto k(v)$ is lower semi-continuous. Let $(v_n)_{n\in \N}$ be some sequence in $U$ converging to some $v\in U$ and let $N\in \N$ with $k_n \coloneqq k(v_n)\leq N$ for all $n\in \N$. Then $f(v_n)=\lambda_{k_n}(v_n) \leq \lambda_N(v_n)$ for all $n\in \N$. By continuity of $\lambda_N$ and lower semi-continuity of $f$, we get that
$$f(v) \leq \liminf_{n\to\infty} f(v_n) \leq \liminf_{n\to \infty} \lambda_N(v_n) = \lambda_N(v).$$
Since $k(v)$ is the smallest $k$ with $\lambda_k(v) = f(v)$, we must have $k(v)\leq N$. We have shown that for all sequences $(v_n)_{n\in \N}\subseteq U$ converging to some $v\in U$ 
$$\{k(v_n)\}_{n\in \N} \subseteq \{1,2,\dots,N\} \implies k(v) \in \{1,2,\dots,N\}.$$
Therefore the function $v\mapsto k(v)$ is lower semi-continuous. For $i\in \N$, let $A_i = \{v\in U \mid k(v)\leq i\}$ and note that $U = \bigcup_{i\in \N} A_i$. By the Baire category theorem, there is some $i\in \N$ such that $A_i$ has non-empty interior. Take $i\in \N$ minimal with this property. Then $V\coloneqq \interior(A_i) \setminus A_{i-1}$ is open. If $\interior(A_i) \setminus A_{i-1} = \emptyset$, then we would have had $\interior(A_i)\subseteq A_{i-1}$ and hence $\interior(A_i)\subseteq \interior(A_{i-1})$, contradicting the minimality of $i$. So $V$ is non-empty. For $v\in V$ we have that $f(v) = \lambda_i(v)$ and by \theoref{thm:NullConjugateLocusIsNice} we can shrink $V$ to some smaller non-empty open set so that $f|_V = \lambda_i|_V$ is smooth. Assume that this $V$ is connected, else take a smaller $V$ that is instead. Let $s\mapsto v_s$ be any $\mathcal{C}^1$ curve in $V$ and let us denote $v=f(v_0)v_0$ and $w=\frac{d}{ds}|_{s=0}(f(v_s)v_s)$. Then, by \lemref{lem:DerivativeOfCurveOfLightlikeVectorsCannotBeTimelike}, 
$\frac{d}{ds}|_{s=0}\exp(f(v_s)v_s) = d\exp_{v}(w)$ is not timelike. But by construction $\exp(f(v_s)v_s) \in \im(\gamma)$ for all $s$, so $\frac{d}{ds}|_{s=0}\exp(f(v_s)v_s)$ is some scalar multiple of a timelike vector. This is only possible if $\frac{d}{ds}|_{s=0}\exp(f(v_s)v_s) = d\exp_v(w)=0$.
Since on $V$ all directional derivatives of $v\mapsto\exp(f(v)v)$ vanish, we conclude that $v\mapsto \exp(f(v)v)$ is constant on $V$.
\end{proof}

\begin{prop}\label{prop:ThereIsAnOpenSetOfNullVectorsWhoseGeodesicsAllIntersectInTheSamePoint}
Let $(X,g)$ be a globally hyperbolic spacetime of dimension at least $3$ which is observer-refocusing with respect to a point $p\in X$ and a timelike curve $\gamma$ in $X$. Let $L$ be a lightsphere at $p$. Then there exists a non-empty open subset $V\subseteq L$ and a point $q\in \im(\gamma)$ such that for all $v\in V$ there is some $t>0$ with $\alpha_v(t)=q$.
\end{prop}
\begin{proof}
As discussed in \remref{rem:WeCanGetopenDenseSetWithoutUniformBound}, there is an open, dense subset $L'\subseteq L$ such that for all $v\in L'$, $p$ has some conjugate point $q$ in the intersection of $\alpha_v$ and $\gamma$. We define $f:L' \to \R^+$ through $\exp(f(v)v) = \text{first conjugate point } q \in \im(\gamma) \text{ along } \alpha_v.$
It is easy to see that $f$ is lower semi-continuous. By \lemref{lem:MapSelectingFirstConjugatePointIsLocallyConstantSomewhere}, $v\mapsto \exp(f(v)v)$ is constant on some non-empty open set $V\subseteq L'$. In particular, there is some $q\in \im(\gamma)$ such that for all $v\in V$ there is some $t=f(v)>0$ such that $\alpha_v(t)=\exp(tv)=q$. 
\end{proof}

If our metric is analytic instead of just smooth, an observer-refocusing spacetime of dimension at least $3$ is always strongly refocusing.

\begin{thm}\label{thm:AnalyticMetricImpliesObserverRefocusingIsStronglyRefocusing}
Let $(X,g)$ be a globally hyperbolic, observer-refocusing spacetime of dimension at least $3$ with an analytic metric $g$. Then $(X,g)$ is strongly refocusing.
\end{thm}
\begin{proof}
By \cite[Theorem 2.10]{Sánchez2022}, there exists an analytic temporal function $\mathcal{T}:X\to \R$ whose level sets are Cauchy surfaces. Let $p\in X$ and let $\gamma$ be a timelike curve in $X$ such that $(X,g)$ is observer-refocusing with respect to $p$ and $\gamma$. Let $L \subset T_pX$ be some lightsphere at $p$ which is also an analytic submanifold of $T_pX$, i.e. a lightsphere such that the inclusion $\iota:L\to T_p X$ is analytic. Let $D \subset L \times \R$ be the maximal domain of the function $(v,t) \mapsto \exp_p(tv)$. Note that the map $H:D\to T_p X$ defined through $H(v,t) = t\cdot \iota(v)$ is analytic. Consider the map 
$$\phi:D \to \R$$
$$\phi(v,t) = \mathcal{T}(\exp_p(t v))$$
Since $g$ is analytic, so is the exponential map $\exp_p:T_pX \to X$. Hence $\phi=\mathcal{T}\circ \exp_p \circ H$ is analytic as the composition of analytic maps.

\medskip
\noindent By \propref{prop:ThereIsAnOpenSetOfNullVectorsWhoseGeodesicsAllIntersectInTheSamePoint} there is some non-empty open set $V\subseteq L$ and a point $q \in \im(\gamma)$ such that for all $v\in V$ there is some $t>0$ with $\alpha_v(t)=q$. Let $T=\mathcal{T}(q)$. Then, for every $v\in L$, there is a unique $t \in \R$ with $(v,t)\in D$ and $\phi(v,t) = T.$ By the analytic implicit function theorem, there is an analytic map $f:L\to \R$ with $(v,f(v))\in D$ and $\phi(v,f(v)) = T$ for all $v\in L$. Now consider the map
$$\psi:L \to X$$
$$\psi(v) = \exp_p(H(f(v),v))=\exp_p(f(v) v)$$
which is analytic since the maps $\exp_p$, $H$, and $f$ are analytic. Note that for $v \in V$ we must have $\psi(v)=q$, so $\psi$ is constant on $V$. Since $L$ is topologically a sphere of dimension $n-2 \geq 1$, $L$ is connected. Since $\psi$ is analytic on all of $L$ and constant on the non-empty open set $V$, $\psi$ is constant on all of $L$.
\end{proof}

\begin{rem}
Note that we did not need to assume that $(X,g)$ is observer-refocusing with respect to an analytic timelike curve $\gamma$. We only need analyticity of the metric $g$. 
\end{rem}

\begin{cor}\label{cor:AnalyticMetricImpliesZxManifoldsAreYxl}
Let $(M,h)$ be a $Z^{(x,y)}$ manifold of dimension at least $2$ with an analytic metric $h$. Then $(M,h)$ is a $Y^{(x,y)}_l$ manifold for some $l>0$. 
\end{cor}
\begin{proof}
This result follows from the proof of \theoref{thm:AnalyticMetricImpliesObserverRefocusingIsStronglyRefocusing} by considering the associated spacetime $(X,g)$ to the Riemannian manifold $(M,h)$, which will be analytic since $(M,h)$ was analytic, defining $p=(x,0)$, and the standard observer $\gamma(t)=(y,t)$.
\end{proof}

In particular, any $Z^x$ manifold with analytic metric is a $Y^x_l$ manifold for some $l>0$ (the $1$-dimensional case is trivially true even without analyticity). Since \corolref{cor:AnalyticMetricImpliesZxManifoldsAreYxl} might be of independent interest, a self-contained proof of this result, avoiding Lorentzian geometry entirely, is supplied in the appendix.

\section{Analogous Conjectures in the Contact Case}

The results of Chernov and Nemirovski \cite{ChernovNemirovski2010}, the results of Frauenfelder, Labrousse and Schlenk \cite[Theorem 1.11 and 1.13]{FrauenfelderLabrousseSchlenk2015}, and our results about Riemannian $Z^{(x,y)}$ manifolds and observer-refocusing spacetimes in this paper suggest the following conjectures.

\begin{conjec}\label{conj:StrongConjecture}
Let $M$ be a manifold of dimension at least $2$, let $(S^{\ast}M,\xi)$ be the spherical cotangent bundle on $M$ equipped with the standard contact structure. Let $\alpha$ be any contact form on $(S^{\ast}M,\xi)$, i.e. $\xi = \ker(\alpha)$, let $\varphi_{\alpha}^t$ be the Reeb flow of $\alpha$ and let $x,y \in M$. Assume that for all $v\in S^{\ast}_xM$ there is some $t>0$ with $\varphi_{\alpha}^t(v) \in S^{\ast}_yM$. Then $M$ is compact, the fundamental group of $M$ is finite, and the integral cohomology ring of the universal cover of $M$ is the one of a CROSS.
\end{conjec}

\begin{conjec}\label{conj:WeakConjecture}
Let $M$, $\xi$, $\alpha$, $\varphi_{\alpha}^t$ be as in \conjref{conj:StrongConjecture}. Let $x,y \in M$. Assume there is some $T>0$ such that for all $v\in S^{\ast}_xM$ there is some $t\in (0,T]$ with $\varphi_{\alpha}^t(v) \in S^{\ast}_yM$. Then $M$ is compact, the fundamental group of $M$ is finite, and the integral cohomology ring of the universal cover of $M$ is the one of a CROSS.
\end{conjec}

This contact-geometric version generalizes the Riemannian $Z^{(x,y)}$ version as follows: Let $(M,g)$ be a Riemannian manifold, let $\psi:TM\to T^{\ast}M$ be the isomorphism defined through $v\mapsto g(v,\cdot)$. Then $\psi$ induces a bundle-isomorphism (which we will also denote by $\psi$) between the unit-tangent bundle $U_gM$ and the spherical cotangent bundle $S^{\ast}M$. We can define a $1$-form $\alpha$ on $S^{\ast}M$ via $\alpha_{v}(w) = g(\psi^{-1}(v),d\pi(w))$ where $\pi:S^\ast M \to M$ is the natural projection. The Reeb flow of $\alpha$ on $S^{\ast}M$ corresponds to the co-geodesic flow of $g$ on $S^{\ast}M$. \conjref{conj:StrongConjecture} generalizes the conjecture that $Z^x$ manifolds of dimension at least $2$ have finite fundamental group. We have proven the compactness and the fundamental group aspect of \conjref{conj:WeakConjecture} whenever the contact form $\alpha$ comes from a Riemannian metric. The conjecture states that this is really a fact about Reeb flows on $(S^{\ast}M,\xi)$ and not just about geodesics on $M$. 

\begin{rem}\label{rem:WeCanSplitXToAccommodateTimelikeCurves}
For a given globally hyperbolic spacetime $(X,g)$ there are many different splittings of $X = M \times \R$ like in \propref{prop:NullGeodesicFlowInducesPositiveLegendrianIsotopy}, each giving rise to positive Legendrian isotopies $\phi_t$ of $S^{\ast}M$. In particular, if $\gamma$ is any timelike curve in $X$, it is always possible to split $X$ as $M\times \R$ such that in these coordinates $\gamma(t)=(y,t)$, i.e. such that $\gamma$ stays constant in the space-coordinate of the chosen coordinate system. Therefore, whenever we have a globally hyperbolic, observer-refocusing spacetime $X$ with Cauchy surface $M$, there will be $x,y \in M$ and a positive Legendrian isotopy $\phi_t$ on $S^{\ast}M$ such that for all $v\in S^{\ast}_xM$ there is some $t>0$ with $\phi_t(v) \in S_y^{\ast}M$. 
\end{rem}

The following two conjectures further generalize \conjref{conj:StrongConjecture} and \conjref{conj:WeakConjecture} by replacing Reeb flows with positive Legendrian isotopies.

\begin{conjec}\label{conj:StrongConjecture2}
Let $M$ be a manifold of dimension at least $2$, let $(S^{\ast}M,\xi)$ be the spherical cotangent bundle on $M$ equipped with the standard contact structure. Let $\phi_t$ be any positive Legendrian isotopy of $(S^{\ast}M,\xi)$ with $\phi_0=\id$ and let $x,y \in M$. Assume that for all $v\in S^{\ast}_xM$ there is some $t>0$ with $\phi_t(v) \in S^{\ast}_yM$. Then $M$ is compact, the fundamental group of $M$ is finite, and the integral cohomology ring of the universal cover of $M$ is the one of a CROSS.
\end{conjec}

\begin{conjec}\label{conj:WeakConjecture2}
Let $M$, $\xi$, and $\phi_t$ be as in \conjref{conj:StrongConjecture2}. Let $x,y \in M$. Assume there is some $T>0$ such that for all $v\in S^{\ast}_xM$ there is some $t\in (0,T]$ with $\phi_t(v) \in S^{\ast}_yM$. Then $M$ is compact, the fundamental group of $M$ is finite, and the integral cohomology ring of the universal cover of $M$ is the one of a CROSS.
\end{conjec}

We have proven the compactness and fundamental group aspect of \conjref{conj:WeakConjecture2} whenever the positive Legendrian isotopy comes from the construction seen in \propref{prop:NullGeodesicFlowInducesPositiveLegendrianIsotopy}. The conjecture states that this is really a fact about positive Legendrian isotopies of $(S^{\ast}M,\xi)$ and not just about the null-geodesic flow in spacetimes which have $M$ as a Cauchy surface.

\section*{Acknowledgments}

The author thanks Vladimir Chernov for the fruitful idea of investigating the topology of Lorentzian analogues of $Z^x$ manifolds. The author also thanks the anonymous reviewer for pointing out errors in an earlier version of this manuscript, and for multiple helpful suggestions to simplify some proofs.

\appendix
\section*{Appendix}
\setcounter{thm}{0}
\renewcommand{\thesection}{A}

In this appendix we present a proof of \corolref{cor:AnalyticMetricImpliesZxManifoldsAreYxl}, the fact that every analytic $Z^{(x,y)}$ manifold of dimension at least $2$ is a $Y^{(x,y)}_l$ manifold for some $l>0$, that does not rely on Lorentzian techniques. This is done for the benefit of people interested only in the Riemannian problem who might be unfamiliar with Lorentzian geometry. The appendix is self-contained and can be read without reference to the rest of the paper. In this section, if $(M,h)$ is a Riemannian manifold, $x\in M$ and $v\in T_xM$, then $\alpha_v$ denotes the maximal geodesic with $\alpha_v(0)=x$ and $\alpha_v'(0)=v$. All Riemannian manifolds are assumed to be connected and complete. We do not give a self-contained proof of \corolref{cor:EveryTildeZxManifoldIsCompact} because the spacetime viewpoint is more natural in this case; in particular there seems to be no obvious way to distill the Lorentzian proof into a purely Riemannian one.

\begin{lem}[Analogue of \lemref{lem:int(intersection)subsetRefocusing}]\label{lem:int(intersection)subsetRefocusingRiemannianCase}
Let $(M,h)$ be a Riemannian manifold of dimension at least $2$. Let $x,y\in M$. Let $L=\{v\in T_xM \mid h(v,v)=1\}$ and let $T>0$ be such that $\alpha_v(T)$ is defined for all $v\in L$. Let $A=\{v\in L \mid y\in \alpha_v((0,T])\}$ and $B=\{v\in L \mid y \text{ is conjugate to } x \text{ along } \alpha_v|_{(0,T]}\}$. Then $\interior(A)\subseteq B$.
\end{lem}
\begin{proof}
Let $v\in \interior(A)$. Since $\dim(M)\geq 2$ implies $\dim(L)\geq 1$, there is some sequence $(v_n)_{n\in \N}$ in $A$ converging to $v$. For each $n\in \N$ let $t_n \in (0,T]$ be such that $\alpha_{v_n}(t_n) = \exp(t_n v_n)=y$. We can assume that $(t_n)_{n\in \N}$ converges to some $t\in (0,T]$, else go to a convergent subsequence. We see that the exponential function fails to be injective on every neighborhood of $tv\in T_xM$ and conclude that $d\exp_{tv}$ is singular. Therefore $y=\alpha_v(t)$ is conjugate to $x$ along $\alpha_v$ and so $v\in B$.
\end{proof}

The following Lemma is analogous to some arguments made in \theoref{thm:MainResult} and \remref{rem:WeCanGetopenDenseSetWithoutUniformBound}.
\begin{lem}\label{lem:AppendixLemmaWithArgumentsAnalogousToThoseOfMainResultAndARemark}
Let $(M,h)$ be a $Z^{(x,y)}$ manifold and let $L=\{v\in T_x M \mid h(v,v)=1\}$. Then there exists an open dense subset $L' \subseteq L$ such that for all $v\in L'$ we have that $y$ is conjugate to $x$ along $\alpha_v$.
\end{lem}

\begin{proof}
Consider the sets $A_N = \{v\in L \mid y\in \alpha_v((0,N])\}$. It is easy to see that each set $A_N$ is closed. By assumption $L=\bigcup_{N\in \N} A_N$. By the Baire category theorem, the set $L' = \bigcup_{N\in \N} \interior(A_N)$ is an open dense subset of $L$, and by \lemref{lem:int(intersection)subsetRefocusingRiemannianCase} for all $v\in L'$ we have that $y$ is conjugate to $x$ along $\alpha_v$.
\end{proof}

\begin{thm}[Analogue of \theoref{thm:NullConjugateLocusIsNice}]\label{thm:RiemannianConjugateLocusIsNice}
Let $(M,h)$ be a Riemannian manifold. Let $x\in M$ and let $L=\{v\in T_xM \mid h(v,v)=1\}$. Let $v\in L$ be such that the geodesic $\alpha_v$ has at least $k$ points conjugate to $x$. Then there is some neighborhood $U\subseteq L$ of $v$ such that for each $w\in U$, the unit-speed geodesic $\alpha_w$ has at least $k$ conjugate points (counted with multiplicity). Let $0<\lambda_1(w) \leq \lambda_2(w)\leq \dots \leq \lambda_k(w)$ be the parameter values at which $\alpha_w$ hits a conjugate point. Then the functions $\lambda_j$ are continuous on $U$ and they are smooth on an open, dense subset of $U$.
\end{thm}
\begin{proof}
The existence of the neighborhood $U$ and the continuity of the functions $\lambda_j$ are essentially the content of \cite[Proposition 5.1]{Warner1965}. The fact that each $\lambda_j$ is smooth on an open dense set of $U$ follows from \cite[Theorem 3.1]{Warner1965}: The set $\{\lambda_j(v)v \mid v\in U\}$ is a subset of the conjugate locus. Let $C\subseteq T_xM$ denote the conjugate locus of the exponential map. The regular conjugate locus as defined in \cite{Warner1965} is an open dense subset $C_R\subseteq C$ in the conjugate locus and it is a smooth submanifold of $T_xM$ transverse to rays in $T_xM$. For every $j\in \{1,\dots,k\}$ the set $V_j\coloneqq \{\lambda_j(v)v\mid v \in U\}$ is an open subset of $C$, hence $V_j \cap C_R$ is open and dense in $V_j$. Since $V_j \cap C_R$ is also an open subset of $C_R$, it is in particular a smooth submanifold of $T_xM$ transverse to the rays in $T_xM$. For each $j$ let $U_j\coloneqq \{\pi(w)\mid w\in V_j\cap C_R\}$ where $\pi:T_xM\setminus\{0\}\to U$ is the radial projection. Each $U_j$ is open and dense in $U$. By construction $\{\lambda_j(v)v \mid v\in U_j\}=V_j\cap C_R$ is a smooth submanifold of $T_xM$ transverse to the rays in $T_xM$. By the implicit function theorem the map $v\mapsto \lambda_j(v)$ is smooth on $U_j$. The intersection of all the sets $U_j$ gives one open, dense subset of $U$ on which all functions $\lambda_j$ are smooth.
\end{proof}

\begin{lem}[Analogue of \lemref{lem:MapSelectingFirstConjugatePointIsLocallyConstantSomewhere}]\label{lem:MapSelectingFirstConjugateTimeIsLocallyConstantSomewhere}
Let $(M,h)$ be a Riemannian manifold of dimension at least $2$. Let $x,y\in M$ and let $L=\{v\in T_xM \mid h(v,v)=1\}$. Let $U\subseteq L$ be some non-empty open set. Let $f:U\to \R^{+}$ be some lower semi-continuous function such that for all $v\in U$ we have that $\exp(f(v)v)=y$ and such that $y$ is conjugate to $x$ along $\alpha_v|_{[0,f(v)]}$. Then there is some non-empty open set $V\subseteq U$ on which $f$ is constant.
\end{lem}
\begin{proof}
Let the functions $\lambda_j$ be as in \theoref{thm:RiemannianConjugateLocusIsNice}. For $v\in U$ we have that $\exp(f(v)v)$ is a conjugate point along $\alpha_v$. Let $k(v)$ denote the minimal $k$ with $\lambda_k(v)=f(v)$. We will show that $v\mapsto k(v)$ is lower semi-continuous. Let $(v_n)_{n\in \N}$ be some sequence in $U$ converging to some $v\in U$ and let $N\in \N$ with $k_n \coloneqq k(v_n)\leq N$ for all $n\in \N$. Then $f(v_n)=\lambda_{k_n}(v_n) \leq \lambda_N(v_n)$ for all $n\in \N$. By continuity of $\lambda_N$ and lower semi-continuity of $f$, we get that
$$f(v) \leq \liminf_{n\to\infty} f(v_n) \leq \liminf_{n\to \infty} \lambda_N(v_n) = \lambda_N(v).$$
Since $k(v)$ is the smallest $k$ with $\lambda_k(v) = f(v)$, we must have $k(v)\leq N$. We have shown that for all sequences $(v_n)_{n\in \N}\subseteq U$ converging to some $v\in U$ 
$$\{k(v_n)\}_{n\in \N} \subseteq \{1,2,\dots,N\} \implies k(v) \in \{1,2,\dots,N\}.$$
Therefore the function $v\mapsto k(v)$ is lower semi-continuous. For $i\in \N$, let $A_i = \{v\in U \mid k(v)\leq i\}$ and note that $U = \bigcup_{i\in \N} A_i$. By the Baire category theorem, there is some $i\in \N$ such that $A_i$ has non-empty interior. Take $i\in \N$ minimal with this property. Then $V\coloneqq \interior(A_i) \setminus A_{i-1}$ is open. If $\interior(A_i) \setminus A_{i-1} = \emptyset$, then we would have had $\interior(A_i)\subseteq A_{i-1}$ and hence $\interior(A_i)\subseteq \interior(A_{i-1})$, contradicting the minimality of $i$. So $V$ is non-empty. For $v\in V$ we have that $f(v) = \lambda_i(v)$ and by \theoref{thm:RiemannianConjugateLocusIsNice} we can shrink $V$ to some smaller non-empty open set so that $f|_V = \lambda_i|_V$ is smooth. Assume that this $V$ is connected, else take a smaller $V$ that is instead. Let $s\mapsto v_s$ be any $\mathcal{C}^1$ curve in $V$ and let us denote $t=f(v_0)$, $v=f(v_0) v_0$, $w=\frac{d}{ds}|_{s=0}(v_s)$, and $t' = \frac{d}{ds}|_{s=0}f(v_s)=df(w)$. Then, since $v\mapsto\exp(f(v)v)$ is constant on $V$, we have
$$0=\frac{d}{ds}|_{s=0}\exp(f(v_s)v_s) = t d\exp_{v}(w) + t' d\exp_{v}(v).$$
This implies $$t d\exp_v(w) = -t' d\exp_v(v).$$
Taking the dot product with $d\exp_v(v)$ of both sides and applying the Gauss lemma yields
$$t h(w,v)=-t' h(v,v) = -t'.$$
The left hand side of this equation is $0$, as $h(w,v)=\frac{1}{2}\frac{d}{ds}|_{s=0}h(v_s,v_s) = \frac{1}{2}\frac{d}{ds}|_{s=0} 1$. Hence $t'=df(w)=0$. Since this holds for all smooth curves $s\mapsto v_s$ in $V$, we have that $df$ vanishes on $TV$, hence $f$ is constant on $V$.
\end{proof}

\begin{thm}[\corolref{cor:AnalyticMetricImpliesZxManifoldsAreYxl}]\label{thm:AppendixVersionAnalyticMetricImpliesZxManifoldsAreYxl}
Let $(M,h)$ be a $Z^{(x,y)}$ manifold of dimension at least $2$ with an analytic metric $h$. Then $(M,h)$ is a $Y^{(x,y)}_l$ manifold for some $l>0$.
\end{thm}
\begin{proof}
Let $L=\{v\in T_xM \mid h(v,v)=1\}$. By \lemref{lem:AppendixLemmaWithArgumentsAnalogousToThoseOfMainResultAndARemark}, there is an open, dense subset $L'\subseteq L$ such that for all $v\in L'$ we have that $y$ is conjugate to $x$ along $\alpha_v$. We define $f:L' \to \R^+$ through
$$f(v) = \inf\{t>0 \mid \alpha_v(t)=y \text{ and } y \text{ is conjugate to } x \text{ along } \alpha_v|_{[0,t]}\},$$
which is lower semi-continuous. By \lemref{lem:MapSelectingFirstConjugateTimeIsLocallyConstantSomewhere} there is an open, non-empty set $V\subseteq L'$ so that $f|_V$ is constant. Let $l$ be the value $f$ takes on $V$, so that for each $v\in V$ we have that $\alpha_v(l)=\exp(lv)=y$. We define 
$$\psi:L\to M, \quad \psi(v)=\exp(lv).$$
Since $h$ is analytic, $L$ is an analytic submanifold of $T_xM$ and $\exp:TM\to M$ is analytic. Therefore $\psi:L\to M$ is analytic. Since $\psi$ is analytic and constant on a non-empty open set $V\subseteq L$, and since $L$ is a sphere of dimension at least $\dim(M)-1\geq 1$ and hence connected, $\psi$ is constant on all of $L$. Therefore we have that $\alpha_v(l)=\exp(lv)=y$ for all $v\in L$, so $(M,h)$ is a $Y^{(x,y)}_l$ manifold.
\end{proof}



\begin{thebibliography}{10}

\bibitem{BautistaIbortLafuente2014}
A.~Bautista, A.~Ibort, and J.~Lafuente, \emph{On the space of light rays of a spacetime and a reconstruction theorem by Low}, 
Classical and Quantum Gravity \textbf{31} (2014), no.~7, 075020.

\bibitem{BeemEhrlich1979}
J.~K. Beem and P.~E. Ehrlich, \emph{Cut points, conjugate points and Lorentzian comparison theorems}, 
Mathematical Proceedings of the Cambridge Philosophical Society \textbf{86} (1979), no.~2, 365--384.


\bibitem{BerardBergery1977}
L.~Bérard-Bergery, \emph{Quelques exemples de variétés riemanniennes où toutes les géodésiques issues d’un point sont fermées et de même longueur, suivis de quelques résultats sur leur topologie}, 
Ann. Inst. Fourier (Grenoble) \textbf{27} (1977), no.~1, 231--249.

\bibitem{BernalSánchez2003}
A.~Bernal and M.~Sánchez, \emph{On smooth Cauchy hypersurfaces and Geroch’s splitting theorem}, 
Communications in Mathematical Physics \textbf{243} (2003), no.~3, 461--470.

\bibitem{BernalSánchez2004}
A. Bernal and M. Sánchez, \emph{Smooth globally hyperbolic splittings and temporal functions}, in
\emph{2nd International Meeting on Lorentzian Geometry}, (2004). 

\bibitem{BernalSánchez2005}
A.~Bernal and M.~Sánchez, \emph{Smoothness of time functions and the metric splitting of globally hyperbolic space-times}, 
Communications in Mathematical Physics \textbf{257} (2005), no.~1, 43--50.

\bibitem{BernalSánchez2006}
A.~Bernal and M.~Sánchez, \emph{Further results on the smoothability of Cauchy hypersurfaces and Cauchy time functions}, 
Letters in Mathematical Physics \textbf{77} (2006), no.~2, 183--197.

\bibitem{BernalSánchez2007}
A.~Bernal and M.~Sánchez, \emph{Globally hyperbolic spacetimes can be defined as “causal” instead of “strongly causal”}, 
Classical and Quantum Gravity \textbf{24} (2007), 745--750.

\bibitem{Besse}
A.~L. Besse, \emph{Manifolds all of whose Geodesics are Closed}, 
Ergebnisse der Mathematik und ihrer Grenzgebiete, vol.~93, Springer-Verlag, Berlin-New York, (1978).

\bibitem{Bott1954}
R.~Bott, \emph{On manifolds all of whose geodesics are closed}, 
Annals of Mathematics \textbf{60} (1954), no.~3, 375--382.

\bibitem{Chernov2018}
V.~Chernov, \emph{Causality and Legendrian linking for higher dimensional spacetimes}, 
Journal of Geometry and Physics, \textbf{133} (2018), 26--29. 

\bibitem{ChernovKinlawSadykov2010}
V.~Chernov, P.~Kinlaw, and R.~Sadykov, \emph{Topological properties of manifolds admitting a $Y^x$-Riemannian metric}, 
Journal of Geometry and Physics, \textbf{60} (2010), no.~10, 1530--1538.

\bibitem{ChernovNemirovski2010}
V.~Chernov and S.~Nemirovski, \emph{Non-negative Legendrian isotopy in $ST^*M$}, 
Geometry \& Topology \textbf{14} (2010), no.~1, 611--626.

\bibitem{ChernovNemirovski2016}
V. Chernov and S. Nemirovski, \emph{Universal orderability of Legendrian isotopy classes}, 
Journal of Symplectic Geometry \textbf{14} (2016), no.~1, 149--170.

\bibitem{ChernovRudyak2008}
V. Chernov and Y.~B. Rudyak, \emph{Linking and Causality in Globally Hyperbolic Space-times}, 
Communications in Mathematical Physics \textbf{279} (2008), 309--354.

\bibitem{Clarke1971}
C.~J.~S.~Clarke, \emph{On the geodesic completeness of causal space-times}, 
Mathematical Proceedings of the Cambridge Philosophical Society \textbf{69} (1971), no.~2, 319--323. 

\bibitem{FrauenfelderLabrousseSchlenk2015}
U.~Frauenfelder, C.~Labrousse, and F.~Schlenk, \emph{Slow volume growth for Reeb flows on spherizations and contact Bott–Samelson theorems}, 
Journal of Topology and Analysis \textbf{7} (2015), no.~3, 407--451.

\bibitem{Geroch1970}
R.~Geroch, \emph{Domain of dependence}, Journal of Mathematical Physics \textbf{11} (1970), 437--449.

\bibitem{HawkingEllis}
S.~W. Hawking and G.~F.~R. Ellis, \emph{The Large Scale Structure of Space-Time}, 
Cambridge University Press, Cambridge, UK, (1973).

\bibitem{Kinlaw2011}
P.~A. Kinlaw, \emph{Refocusing of light rays in space-time}, 
Journal of Mathematical Physics \textbf{52} (2011), no.~5, 052505.

\bibitem{Low2006}
R.~J. Low, \emph{The space of null geodesics (and a new causal boundary)}, 
Lecture Notes in Physics \textbf{692}, Springer, Berlin Heidelberg New York, (2006), 35--50.

\bibitem{Penrose1998}
R.~Penrose, \emph{The question of cosmic censorship}, in
\emph{Black Holes and Relativistic Stars} (Chicago, IL, 1996), University of Chicago Press, Chicago, IL, 1998, pp.~103--122.

\bibitem{Perelman2002}
G.~Perelman, \emph{The entropy formula for the Ricci flow and its geometric applications}, 
Preprint math.DG/0211159 (2002), available at \url{https://arxiv.org/abs/math/0211159}.

\bibitem{Perelman2003}
G.~Perelman, \emph{Ricci flow with surgery on three-manifolds}, 
Preprint, math.DG/0303109 (2003), available at \url{https://arxiv.org/abs/math/0303109}.

\bibitem{Perelman2003-2}
G.~Perelman, \emph{Finite extinction time for the solutions to the Ricci flow on certain three-manifolds}, 
Preprint, math.DG/0307245 (2003), available at \url{https://arxiv.org/abs/math/0307245}.

\bibitem{Rosquist1983}
K.~Rosquist, \emph{On the structure of space-time caustics}, 
Communications in Mathematical Physics \textbf{88} (1983), no.~3, 339--355.

\bibitem{Samelson1963}
H.~Samelson, \emph{On manifolds with many closed geodesics}, 
Portugaliae Mathematica \textbf{22} (1963), 193--196.

\bibitem{Sánchez2022}
M.~Sánchez, \emph{Globally hyperbolic spacetimes: slicings, boundaries and counterexamples}, 
General Relativity and Gravitation \textbf{54} (2022), no.~10, 124.

\bibitem{Thurston1997}
W.~Thurston, \emph{Three-dimensional Geometry and Topology. Vol. 1}, edited by Silvio Levy, 
Princeton Mathematical Series, vol.~35, Princeton University Press, Princeton, NJ (1997).

\bibitem{Warner1965}
F.~W. Warner, \emph{The Conjugate Locus of a Riemannian Manifold}, 
American Journal of Mathematics \textbf{87} (1965), no.~3, 575--604.


\end{thebibliography}
\end{document}